\documentclass[11pt]{elsarticle}
\usepackage[a4paper,margin=3cm]{geometry}
\usepackage{times}
\usepackage{amsthm}
\usepackage{amsmath}
\usepackage{amssymb}
\usepackage{mathtools}
\usepackage{relsize}%
\usepackage[T1]{fontenc}
\usepackage[utf8]{inputenc}
\usepackage{enumerate}
\usepackage{float}
\usepackage{verbatim}

\usepackage{tikz}
\usetikzlibrary{decorations,patterns}
\theoremstyle{plain}
\newtheorem{theorem}{Theorem}[section]
\newtheorem*{theorem*}{Theorem}
\newtheorem{lemma}[theorem]{Lemma}
\newtheorem{corollary}[theorem]{Corollary}
\newtheorem*{corollary*}{Corollary}

\newtheorem{remark}[theorem]{Remark\textbf{}}
\newtheorem*{remark*}{Remark}
\newtheorem{proposition}[theorem]{Proposition}
\newtheorem{conjecture}[theorem]{Conjecture}
\newtheorem{problem}[theorem]{Problem}

\theoremstyle{definition}
\newtheorem{definition}[theorem]{Definition}

\def\mD{\mathbb{D}}

\def\deg{\mathrm{deg}}
\def\ds{\rule{0pt}{1.5ex}}

\newcommand\DOI[1]{{DOI:#1}}
\newcommand\junk[1]{}

\newcommand{\dd}[3]{\operatorname{\mathbb D}(#1,#2,#3)}

\newcommand{\dds}[4]{\operatorname{\mathbb D}(#1,#2,#3,#4)}

\newcommand{\lb}{\operatorname{LEG}}

\newcommand{\pp}[1]{\partial(#1)}

\newcommand{\DD}[1]{\mathcal D^{#1}}

\newcommand{\lex}{\operatorname{\triangleleft_{lex}}}

\newcommand{\simple}{primitive}
\newcommand{\FG}{fully graphic}

\newcommand{\eg}{Erdős-Gallai}
\newcommand{\cone}{cone region}
\usepackage{xcolor}
\newcommand{\jms}{\varphi^*_{J\!M\!S}}
\newcommand{\jmsw}{\varphi^\circ_{J\!M\!S}}

\begin{document}

\begin{frontmatter}
\title{Fully graphic degree sequences and P-stable degree sequences\tnoteref{t1}}
    \tnotetext[t1]{The authors thank to Gábor Lippner the
    interesting discussions about the topic}
\author[renyi]{Péter L. Erdős\fnref{elp}}
\author[renyi]{István Miklós\fnref{elp}}
\author[renyi]{Lajos Soukup\fnref{soukup}}
\address[renyi]{HUN-REN Alfréd Rényi Institute of Mathematics, Reáltanoda u 13--15 Budapest, 1053 Hungary\\
\texttt{$<$erdos.peter,miklos.istvan,soukup.lajos$>$@renyi.hun-ren.hu}}
\fntext[elp]{The authors were supported in part by NKFIH grant SNN~135643, K~132696}
\fntext[soukup]{LS was supported in part by NKFIH grant K~129211.}

\begin{abstract}
The notion of $P$-stability of an infinite set of degree sequences plays influential role in  approximating the permanents, rapidly sampling the realizations of graphic degree sequences, or even studying and improving network privacy. While there exist several known sufficient conditions for $P$-stability, we don't know any useful necessary condition for it. We also do not have good insight of possible structure of $P$-stable degree sequence families.

At first we will show that every known infinite $P$-stable degree sequence set, described by inequalities of the parameters $n, c_1, c_2, \Sigma$ (the sequence length, the maximum and minimum degrees and the sum of the degrees) is ,,fully graphic" meaning that   every degree sequence from the region with an even degree sum, is graphic. Furthermore, if $\Sigma$ does not occur in the determining inequality, then the notions of $P$-stability and full graphicality will be proved equivalent. In turns, this equality provides a strengthening of the well-known theorem of Jerrum, McKay and Sinclair about $P$-stability, describing the maximal $P$-stable sequence set by $n, c_1, c_2$. Furthermore we conjecture that similar equivalences occur in cases if $\Sigma$ also part of the defining inequality.

\end{abstract}

\end{frontmatter}

\section{Introduction}\label{sec:intro}

In this paper we explore the connection  between P-stability and the ``fully graphic'' property of specific regions of degree sequences, establishing  a close relationship between these concepts.

We start the section with recalling and introducing some notions and notations.
If $n>c_1\ge c_2$ and $\Sigma$ are natural numbers with $n\cdot c_1\ge \Sigma \ge n\cdot c_2$, let
\begin{align*}
&\dd{n}{c_1}{c_2}=\{(d_1,\dots,d_n)\in \mathbb N^n: c_1\ge d_1\ge\dots \ge d_n\ge c_2, \textstyle \sum_{i=1}^nd_i \text{ is even}\},\\\medskip
&\dds{n}{\Sigma}{c_1}{c_2}=\{(d_1,\dots,d_n)\in \dd{n}{c_1}{c_2}:\
\textstyle\sum_{i=1}^nd_i=\Sigma\}.
\end{align*}
We will refer the elements of $\dd{n}{c_1}{c_2}$ as \textbf{degree sequences}. As usual, an element of
$\dd{n}{c_1}{c_2}$ is \textbf{graphic} if there is a simple graph with these degrees. Otherwise the element is \textbf{non-graphic}. We say that a  collection $\mathbb D $ of  degree sequences is  a \textbf{simple degree sequence region} if and only if there exists a property $\varphi (n, \Sigma, c_1,c_2)$ such that
\begin{equation}\label{eq:simpleset}\tag{SR}
\mathbb D=\mathbb D[\varphi]\coloneqq\bigcup\left\{\mD(n,\Sigma,c_1,c_2)\ :\
n>c_1\ge c_2 \text{ and } \varphi (n, \Sigma, c_1,c_2)\text{ holds}
\right\}.
\end{equation}
Similarly, we say that a  collection $\mathbb D $ of  degree sequences is a \textbf{very simple degree sequence region} if and only if there exists a property  $\psi (n, c_1,c_2)$ such that
\begin{equation}\label{eq:verysimpleset}\tag{VSR}
\mathbb D=\mathbb D[\psi]\coloneqq\bigcup\left\{\mD(n,c_1,c_2)\ :\
n>c_1\ge c_2 \text{ and } \psi (n,  c_1,c_2)\text{ holds} \right\}.
\end{equation}
We will use the expression \textbf{(very) simple region} for short. \medskip

\noindent
A (very) simple  region $\mathbb D$  will be called  \textbf{fully graphic} if and only if every degree sequence from the region is graphic. $\mathbb D$ is \textbf{almost fully graphic} if and only if $\mathbb D\setminus \mathcal D$ is fully graphic for some finite $\mathcal D$.

Given a graphic degree sequence ${D}$ of length $|D|=n$ let
\begin{equation}\label{eq:localP}
\pp {{D}}=\sum_{1\le i< j\le n}{|\mathcal{G}({D}+1^{-i}_{-j})|}/{|\mathcal{G}({D})|}.
\end{equation}
where the vector $1^{-i}_{-j}$  is comprised of all zeros, except at the $i$th and $j$th coordinates, where the values are $-1$. The operation $D\mapsto {D}+1^{-i}_{-j}$  is called  a \textbf{perturbation operation} on the degree sequences. We define two more similar operations: $D+1^{-i}_{+j}$  and $D+1^{+i}_{+j}$ analogously. Let us emphasize that we do not assume that $i$ and $j$ are different, so for example the operation $D+1^{+j}_{+j}$ is also  defined, and it adds 2 to the $j^{th}$ coordinate.

We say that a family $\mD$ of degree sequences  is {\bf P-stable} if there is a  polynomial $p(n)$ such that $\pp{{D}}\le p(|{D}|)$ for each graphic element $D$ of $\mD$. Let us emphasize that we do not require   every element of a P-stable family  to be  graphic. The P-stability has alternative, but equivalent definitions  using different  perturbation operations. In the Appendix of this paper, we describe a short history of these definitions, and prove their equivalence. (As far as we are aware, this is the first explicit discussion of this topic in writing.)

The  P-stability of an infinite  family of degree sequences is an important property. It would be  enough to mention one result:  the \textbf{switch Markov Chain} process rapidly mixes on  P-stable families (see \cite[Theorems 8.3]{P-stable}).  Furthermore that P-stability plays an influential role in  approximating the permanents, and  even studying and improving network privacy.

Determining whether a particular family is P-stable is typically challenging. There exist only a few results  establishing the P-stability of families of degree sequences for simple graphs. First, we examine three of them which implicitly considered  simple and very simple regions.

\begin{enumerate}[{\rm (P1)}]
\item (Jerrum, McKay, Sinclair \cite{JMS92}) The very simple  degree sequence region $\mathbb D[{\varphi}_{\ds \! {J\!M\!S}}]$ is P-stable, where $\varphi_{\ds \! {J\!M\!S}}\equiv (c_1-c_2+1)^2 \le 4 c_2(n-c_1-1)$.
\item (Greenhill, Sfragara \cite{G18}) The  simple degree sequence region $\mathbb D[{\varphi}_{\ds \! {GS}}]$ is P-stable, where  $\varphi_{\ds \! {GS}}\equiv (2\le c_2\text{ and }  3\le c_1 \le \sqrt{\Sigma/9})$. (This result was not announced explicitly, but \cite[Lemma 2.5]{G18} clearly proved this fact.)
\item (Jerrum, McKay, Sinclair \cite{JMS92}) The simple  degree sequence region $\mathbb D[\jms]$ is P-stable, where $    {\jms} \equiv (\Sigma-nc_2) (nc_1-\Sigma)\le (c_1-c_2)\left\{(\Sigma-nc_2)(n-c_1-1)+ (nc_1-\Sigma)c_2\right\}.$
\end{enumerate}

In contrast with P-stability,  the classical  theorem of Paul Erdős and Gallai (\cite{EG}) makes  easy to check if a certain sequence is graphic or not.  As an extension of this result,  in Section \ref{sec:non-graphic} we show that every  simple region $\dds{n}{\Sigma}{c_1}{c_2}$ contains exactly one {\bf primitive} element, (i.e.a sequence  in the form $(c_1,\dots, c_1, a,c_2,\dots c_2)$),  and the region $\dds{n}{\Sigma}{c_1}{c_2}$ is fully graphic if and only if its primitive element is graphic   (see Theorem \ref{tm:SIMPLE}).

In this way we have established a machinery to decide whether   certain (very) simple regions are fully graphic or not.   Using that machinery, in Theorems \ref{th:JMS}, \ref{th:GS} and \ref{th:JMS2} we show the following
\begin{enumerate}[{\rm (P1${}^*$)}]
\item The  very simple  region 	$\mathbb D[{\varphi}_{\ds \! {J\!M\!S}}]$ is fully graphic.
\item The simple region $\mathbb D[{\varphi}_{\ds \! {GS}}]$ is fully graphic.
\item The simple region $\mathbb D[\jms]$ is fully graphic.
\end{enumerate}
When comparing statements (P1) and (P2) with their counterparts (P1${}^*$) and (P2${}^*$), a fundamental question arises, serving as the central focus of this paper: ``{\em What is the connection between P-stability and the fully graphic property  of specific (very) simple regions?}''

Concerning the ``{\em fully graphic $\longrightarrow$ P-stable}'' direction, if we consider very simple regions, then we can prove the following strengthening of (P1) which is actually the strongest possible result (Theorem \ref{tm:GR2PIntr}):
\begin{theorem*}
The largest  fully graphic very simple region
\begin{displaymath}
\mD_{\max}\coloneqq\bigcup\{\dd{n}{c_1}{c_2}:\text{ $\dd{n}{c_1}{c_2}$ is fully graphic}\}
\end{displaymath}
is  $P$-stable, and so the switch Markov chain is rapidly mixing on $\mD_{\max}$.
\end{theorem*}
We do not know whether  a similar statement holds for simple regions.
\begin{problem}
Is it true that the largest fully graphic simple region 	
\begin{displaymath}
\mD'=\bigcup\{\dds{n}{\Sigma}{c_1}{c_2}:\text{ $\dds{n}{\Sigma}{c_1}{c_2}$ is fully graphic}\}
\end{displaymath}
is  $P$-stable?
\end{problem}

Concerning the ``{\em P-stable $\longrightarrow$ fully graphic }'' direction,
we have  partial  result claiming that a P-stable very simple region should contain
large fully graphic very simple regions provided the region satisfies some natural
restrictions. To formulate our result precisely, let us say that a very simple region $\mathbb D$ is a {\bf \cone} if and only if for some functions  $f,g\in \mathbb N^{\mathbb N}$ we have
\begin{displaymath}
\mathbb D=\mathbb D(f,g)\coloneqq\bigcup\{\dd{n}{g(n)}{f(n)}:n\in \mathbb N\}.
\end{displaymath}
In Section 	\ref{sec:P-graphic} we prove the following results
(Theorem \ref{tm:P2FG} and Corollary \ref{cor:P2Gr})	
\begin{theorem*}
Assume that $f,g,h\in \mathbb N^{\mathbb N}$ are increasing functions .
If the  \cone\ $\mathbb D(f,g)$ is P-stable, then $\mathbb D(f+h,g-h)$ is almost fully graphic provided
\begin{enumerate}[{\rm (1)}]
\item $f(n+k)\le f(n)+k$ for each $n,k\in \mathbb N$,
\item $lim_{n\to \infty}h(n)/\ln(n)=\infty$.
\end{enumerate}
\end{theorem*}

\begin{corollary*}
Assume that $0\le {\varepsilon}_2<{\varepsilon}'_2<{\varepsilon}'_1< {\varepsilon}_1\le 1$. 	
If the very simple region $\mD\coloneqq\bigcup_{n\in {\mathbb N}}\dd{n}
{\lfloor {\varepsilon}_1\cdot n \rfloor }{\lceil {\varepsilon}_2\cdot n\rceil }
$ is P-stable, then the region $\mD'=\bigcup_{n\in {\mathbb N}}\dd{n}
{\lfloor {\varepsilon}'_1\cdot n \rfloor }{\lceil {\varepsilon}'_2\cdot n\rceil }
$ is almost fully graphic.
\end{corollary*}
\noindent This is a partial answer for the problem raised in the Abstract: the structure of each very simple $P$-stable regions is essentially fully determined. The proof of the  previous theorem is  based on some observations concerning split graphs and Tyshkevich product.

\medskip
In Section \ref{sec:stable}, besides proving that the largest graphic very simple region is P-stable,  we also construct new P-stable simple regions.  To do so we should improve a method  which was developed in \cite{JMS92} to prove (P1). In \cite{JMS92} to obtain (P1) Jerrum, McKay and Sinclair actually proved that $\pp{{D}}\le n^{10}$ for each  ${{D}}\in \dd{n}{c_1}{c_2}$ provided ${\varphi}_{\ds \! {J\!M\!S}}(n,c_1,c_2)$ holds. Using a finer analysis we can estimate $\pp{{D}}$ for a single degree sequence under milder assumptions. Namely,  we will prove the following statement  (Theorem \ref{tm:jms-plusss}):

\begin{theorem*}
If  a graphic degree sequence	 ${D}=(d_1, d_2, \dots, d_n)$ satisfies
\begin{equation}\label{eq:weaker-revi}\tag{\ref{eq:weaker-revisitedd}}
\forall k\in [1,n] \quad  \sum_{i=1}^k d_i\le k\cdot (k-1)+d_n\cdot(n-k)+1,
\end{equation}
then  $\pp{{D}}\le 3\cdot n^{9}$.
\end{theorem*}

In \cite{JMS92} the authors proved that $\mD[{\varphi}_{\ds \! {J\!M\!S}}]\subset \mD[{\jms}]$, i.e.  (P3) is a stronger statement than  (P1). However,  the following result (Theorem \ref{tm:d_0}) shows that  there exists simple $P$-stable regions which is completely disjoint from  (P3):
	
\begin{theorem*}
The simple region
\begin{equation}
\mD_0=\bigcup\{\dds{2m}{4m}{m}{1}:m\ge 4\}
\end{equation}
is fully graphic and  P-stable,  but $\mD_0\cap \mD[\varphi^*_{J\!M\!S}]=\emptyset$.	\end{theorem*}

\section{Characterization of \FG \ regions $\dds{n}{\Sigma}{c_1}{c_2}$} \label{sec:non-graphic}
In this section we study fully graphic simple degree sequence regions.

\begin{definition}\label{df:EG}
A non-increasing sequence $(d_1,\dots, d_n)$  has the {\em Erdős-Gallai property} if and only if for each $1\le k\le n$,
\begin{displaymath}\tag{EG${}_k$}\label{eq:EG-prop}
\sum_{i=1}^kd_i\le k(k-1)+\sum_{i=k+1}^n\min(d_i,k).
\end{displaymath}
\end{definition}
\noindent	
The next statement is almost trivial, but will be very convenient later on.
\begin{proposition}\label{th:trivi}
For all $D\in \dds{n}{\Sigma}{c_1}{c_2}$, property (\ref{eq:EG-prop}) holds for each $k\le c_2$ and $k> c_1$.
\end{proposition}
\begin{proof}
When $k\le c_2$ then
$$
\sum_{i=1}^kd_i\le k c_1\le k(n-1)=k(k-1)+(n-k)k.
$$
If $k > c_1$ then $kc_1\le k(k-1).$
\end{proof}

We denote by $(a)_m$  the constant   $a$ sequence of length $m$.
\begin{definition}\label{df:simple}
If $n>c_1\ge c_2$ are natural numbers, we say that a sequence
$D=(d_1,\dots, d_n)\in \dds{n}{c_1}{c_2}$ is {\bf\simple} if and only if $D$ has the form
\begin{equation}\label{eq:simple}
D= \underbrace{c_1,\ldots,c_1}_k, a, \underbrace{c_2, \ldots, c_2}_{(n-1)-k},
\end{equation}
for some  $1\le k\le n $ and $c_2\le a\le c_1$.
\end{definition}

\begin{definition}\label{df:MBLB}
If $n>c_1\ge c_2$ and $\Sigma$ are natural numbers with $n\cdot c_1\ge \Sigma \ge n\cdot c_2$, we define the \emph{least \eg\ sequence $\lb(n,\Sigma,c_1,c_2)$} of length $n$ with the given sum $\Sigma$ and with given maximum and minimum elements $c_1$ and $c_2$ as follows:\\
If $c_1=c_2$, then let $\lb(n,\Sigma,c_1,c_2)=(c_1)_n$. If $c_2<c_1$, then let
\begin{equation}\label{eq:formLB}
	\lb(n,\Sigma,c_1,c_2)= \underbrace{c_1,\ldots,c_1}_{\alpha}, a, \underbrace{c_2, \ldots, c_2}_{(n-1)-\alpha},
	\end{equation}
where
\begin{equation}\label{eq:alpha}
\alpha=\left \lfloor \frac{\Sigma - n\cdot c_2}{c_1-c_2}\right \rfloor
\end{equation}
and
\begin{equation}\label{eq:a}
a = \Sigma -{\big (}{\alpha}\cdot c_1+ (n-1-{\alpha})\cdot c_2{\big )}.
\end{equation}
\end{definition}

\begin{theorem}\label{tm:SIMPLE}
Assume that $n>c_1\ge c_2$ and $\Sigma$ are natural numbers with $n\cdot c_1\ge \Sigma \ge n\cdot c_2$, $\Sigma$ is even.
\begin{enumerate}[{\rm (1)}]
\item $\lb(n,\Sigma,c_1,c_2)\in \dds{n}{\Sigma}{c_1}{c_2}$ and it is \simple.
\item If $\lb(n,\Sigma,c_1,c_2)$ is \eg, then every element of $\dds{n}{\Sigma}{c_1}{c_2}$ is \eg.
\end{enumerate}
\end{theorem}

\begin{proof}[Proof of Theorem \ref{tm:SIMPLE}\,(1)] Clearly, $\sum \lb(n,\Sigma,c_1,c_2)=\Sigma,$ and $\lb(n,\Sigma,c_1,c_2)$ is \simple\ by its definition.

Since  ${\alpha}\cdot c_1+ (n-1-{\alpha})\cdot c_2= {\alpha}\cdot (c_1-c_2)+(n-1)\cdot c_2$, we have
\begin{equation}\label{eq:age}
a\ge \Sigma - \frac{\Sigma - n\cdot c_2}{(c_1-c_2)}(c_1-c_2)+ (n-1)\cdot c_2=c_2.
\end{equation}
Moreover, $({\Sigma - n\cdot c_2})/{(c_1-c_2)}-1<{\alpha}$, hence
\begin{align}
a=\Sigma  &- ({\alpha}\cdot (c_1-c_2)+(n-1)\cdot c_2)\le \nonumber\\
&\le \Sigma -\left(\left(\frac{\Sigma - n\cdot c_2}{c_1-c_2}-1\right) \cdot(c_1-c_2)\right) +(n-1)\cdot c_2 =c_1.\label{eq:ale}
\end{align}
Putting together  \eqref{eq:age} and \eqref{eq:ale}, we obtain $c_1\ge a\ge c_2$.
So $\lb(n,\Sigma,c_1,c_2)$ is really an element of $\dd{n}{c_1}{c_2}$.
\end{proof}
\noindent
Before proving Theorem \ref{tm:SIMPLE}\,(2), we need some preparation. Denote $\lex$ the lexicographical order on finite sequences (i.e. $(d_1,\dots, d_n)\lex(e_1,\dots, e_n)$ if and only if $d_j<e_j$, where $j=\min\{i:d_i\ne e_i\}$).

\begin{lemma}\label{lm:one-step}
Assume that  $D=(d_1,\dots,d_n )\in \dds{n}{\Sigma}{c_1}{c_2}$, and  $1\le \ell<m\le n$ with $d_\ell<c_1$ and $d_m>c_2$. Write $D'=D+1^{+\ell}_{-m}$.
\begin{enumerate}[{\rm (1)}]
\item $D'\in \dds{n}{\Sigma}{c_1}{c_2}$ and $D\lex D'$.
\item  If $D'$ is \eg, then so is $D$.
\end{enumerate}
\end{lemma}

\begin{proof}
(1) It is trivial from the construction.

\noindent (2)
Assume that   \eqref{eq:EG-prop}  fails for $D$:
$\qquad\displaystyle 	\sum_{i=1}^kd_i> k(k-1)+\sum_{i=k+1}^n\min(d_i,k).$ \\
We will show that \eqref{eq:EG-prop}  fails for $D'=(d'_1,\dots,d'_n)$. We should distinguish three cases:

\smallskip
\noindent {\bf Case 1}. $k<\ell.$

If $\min(k,d_\ell)<\min(k,d'_\ell)$, then $\min(k,d_\ell)+1=\min(k,d_\ell+1)$, and so $d_\ell<k$. Hence, $\min(d'_m,k)=\min(d_m-1,k)=d_m-1$. Thus,  $\min(k,d'_\ell)+\min(k,d'_m)= \min(k,d_\ell)+\min(k,d_m)$.
If $\min(k,d_\ell)=\min(k,d'_\ell)$, then $\min(k,d'_\ell)+\min(k,d'_m) \le\min(k,d_\ell)+\min(k,d_m)$.

Thus, $\sum_{i=k+1}^n\min(k,d_i)\ge \sum_{i=k+1}^n\min(k,d'_i)$, and so
\begin{equation}
\sum_{i=1}^kd'_i=\sum_{i=1}^k{d_i}>k(k-1)+\sum_{i=k+1}^n\min(k,d_i)\ge
k(k-1)+\sum_{i=k+1}^n\min(k,d'_i).
\end{equation}

\smallskip \noindent {\bf Case 2.} $\ell \le k<m$.

Then $d_i'\ge d_i$ for $i\le k$ and $d'_i\le d_i$ for $k+1\le i\le n$, so
\begin{equation}
\sum_{i=1}^kd'_i\ge \sum_{i=1}^k{d_i}>k(k-1)+\sum_{i=k+1}^n\min(k,d_i)\ge
k(k-1)+\sum_{i=k+1}^n\min(k,d'_i).
\end{equation}	

\smallskip \noindent {\bf Case 3.} $m\le k$.

Since $\ell < m$, $d'_\ell=d_\ell+1$ and $d'_m=d_m-1$, we have $\sum_{i=1}^kd'_i=\sum_{i=1}^kd_i+1-1 =\sum_{i=1}^kd_i$. Moreover,  $d_i'=d_i$ for $k+1\le i\le n$. Thus
\begin{equation}
\sum_{i=1}^kd'_i=\sum_{i=1}^k{d_i}>k(k-1)+\sum_{i=k+1}\min(k,d_i)=
k(k-1)+\sum_{i=k+1}\min(k,d'_i).
\end{equation}	
Hence, \eqref{eq:EG-prop} really   fails for $D'$.
\end{proof}

\begin{lemma}\label{lm:lexmax}
$\dds{n}{\Sigma}{c_1}{c_2}$ contains just one \simple\ element,  $\lb(n,\Sigma,c_1,c_2)$, which is the $\lex$-maximal element of 	$\dds{n}{\Sigma}{c_1}{c_2}$.
\end{lemma}

\begin{proof} We know that 	$\lb(n,\Sigma,c_1,c_2)$ is a \simple\ element of $\dds{n}{\Sigma}{c_1}{c_2}$ 	by Theorem \ref{tm:SIMPLE} (1). Observe now that $\dds{n}{\Sigma}{c_1}{c_2}$ can not contain two different \simple\ elements.

Indeed, assume that $A=(c_1)_ka(c_2)_{n-k-1}$ and $B=(c_1)_\ell b(c_2)_{n-\ell-1}$ 	are from  $\dds{n}{\Sigma}{c_1}{c_2}$. If $\Sigma=c_1\cdot n$, then $\dds{n}{\Sigma}{c_1}{c_2}$ contains just one element: the $(c_1)_n$ sequence. So we can assume that $\Sigma<n\cdot c_1$, and so we can assume that $a,b<c_1$.

If $k=\ell$, then $a=\Sigma-(kc_1+(n-k-1)c_2)=\Sigma-(\ell c_1+(n-\ell-1)c_2)=b$, so the two sequences are the same. Assume that $k<\ell$. Using that $a<c_1$, we obtain
\begin{equation}\label{eq:simple-unique}
\sum A = kc_1+a+(n-k-1)c_2<(k+1)c_1+(n-k-1)c_2\le \ell c_1+b+ (n-\ell-1)c_2=\sum B,
\end{equation}
contradiction. We proved the observation.

The $\lex$-maximal element of $\dds{n}{\Sigma}{c_1}{c_2}$ is \simple\
by Lemma \ref{lm:one-step}(1), 	so it should be  $\lb(n,\Sigma,c_1,c_2)$.
\end{proof}

\begin{proof}[Proof of Theorem \ref{tm:SIMPLE}\,(2)]
Assume that the set
$$
\mathcal D=\{D\in \dds{n}{\Sigma}{c_1}{c_2}: D \text{ is not \eg}\}
$$
is not empty. Let $D^*$ be the $\lex$-maximal element of $\mathcal D$.
If $D^*$ is not \simple\ then let $\ell=\min\{i:1\le i\le n: c_2<d_i<c_1\}$
and $m=\max\{i:1\le i\le n: c_2<d_i<c_1\}$. Then  $\ell<m$ and so  $D^*+1^{+\ell}_{-m}\in \mathcal D$ and  $D^*\lex D^*+1^{+\ell}_{-m}$ by Lemma \ref{lm:one-step}. Contradiction, and so the maximal element of $\mathcal D$ is \simple. So, by Lemma \ref{lm:lexmax}, $\lb(n,\Sigma,c_1,c_2)$ is in $\mathcal D$.
\end{proof}	

\begin{corollary}\label{th:check-simple}
A simple region  $\mD$ is \FG\ if and only if
$\lb(n,\Sigma,c_1,c_2)$ is graphic whenever
$\dds{n}{\Sigma}{c_1}{c_2}\ne \emptyset$, and so $\dds{n}{\Sigma}{c_1}{c_2}\subset\mD$.
\end{corollary}
\begin{proof}
If $D\in \mD$ is not graphic, then
$D\in \dds{n}{\Sigma}{c_1}{c_2}\subset \mD$ for some parameters $n,\Sigma,c_1,c_2$,
and $\lb(n,\Sigma,c_1,c_2)$ is not \eg\ by Theorem \ref{tm:SIMPLE}(2).
\end{proof}

\medskip\noindent
The following, easy to prove statement will simplify some arguments in Section \ref{sec:stable}.
\begin{lemma}\label{th:surprise}
Let $\dd{n}{c_1}{c_2}$ be fully graphic. Then,
${\varphi}_{\ds \! {FG}}(n,c_2,c_1)$ holds, where
\begin{equation}\label{eq:FG}
	{\varphi}_{\ds \! {FG}}(n,c_1,c_2) \ \equiv \ \forall k\in [1,n] \quad  c_1\cdot k \le k\cdot (k-1)+c_2\cdot(n-k)+1.
\end{equation}
\end{lemma}
\noindent
Indeed, fix $k$. If $\dd{n}{c_1}{c_2}$ is fully graphic, then either $c_1k+(n-k)c_2$ or $c_1k+(n-k)c_2+1$ is even,  so either the sequence $(c_1)_k(c_2)_{n-k}$, or the  sequence $(c_1)_k(c_2+1)(c_2)_{n-k-1}$ is in $\dd{n}{c_2}{c_1}$, and so it is graphic. We apply the Erd\H os-Gallai theorem. In the first case the inequality holds with 1 surplus. In the second case the displayed inequality holds.

\section{The known $P$-stable, simple degree sequence regions are fully graphic }\label{sec:always}
In this section we will show that some known $P$-stable degree sequence regions are also \FG.

\subsection{Sequences defined by minimum and maximum degrees}\label{sec:alw}
Our first result  below clearly implies the statement (P1${}^*$) from the Introduction. In the proof we will use  the machinery we developed in Section  \ref{sec:non-graphic}.
\begin{theorem}\label{th:JMS}
If $n>c_1\ge c_2$ are natural numbers such that
\begin{equation}\label{eq:ineq-JSMnewSL}
(c_1-c_2+1)^2 \le 4c_2(n-c_1-1),
\end{equation}
then every sequence from $\dd{n}{c_1}{c_2}$ has the  \eg\ property.
\end{theorem}
\noindent Before proving this result we have to point out that this result was already proved by Zverovich and Zverovich (\cite{ZZ}) in 1992. It was somewhat strengthened in \cite{CMN} by Cairns, Mendan and Nikolayevsky. Actually they used the inequality $4nc_2\ge (c_1+c_2+1)^2$ and neither paper recognized that this is identical with \eqref{eq:ineq-JSMnewSL}. However, our proof is new, and as construction \eqref{eq:simple_d0} shows, it also provides a slightly larger always graphic region.

Before we continue our proof we should recall the following theorem. As the authors remarked in the first lines  of the proof of  the ``Theorem'' in \cite{TV03}, they actually proved the following statement.

\begin{theorem}[Tripathi-Vijay \cite{TV03}]\label{tm:TV+}
Assume that $D=(d_1,d_2,\dots, d_n)$ is a non-increasing sequence of non-negative integers  and $n>d_1$. Then the following two statements are equivalent:
\begin{enumerate}[{\rm (1)}]
\item $D$ has the \eg\ property,
\item $\eqref{eq:EG-prop}$ holds for $k\in \{j:1\le j\le n \text{ and } d_j>d_{j+1}\}$.
\end{enumerate}  		
\end{theorem}

\begin{proof}[Proof of  of Theorem \ref{th:JMS}]
Assume on the contrary that $D=(d_1,\dots,d_n)\in \dd{n}{c_1}{c_2}$ is not \eg. Write $\Sigma=\sum_{i=1}^nd_i$. By Theorem \ref{tm:SIMPLE}(2),  $E=\lb(n,\Sigma,c_1,c_2)\in \dd{n}{c_1}{c_2}$ is not \eg, as well. $E$ has the form $(c_1)_ka(c_2)_{n-1-k}$, where $c_1\ge a\ge  c_2$. Applying Theorem \ref{tm:TV+} for $E$ we obtain that either (EG${}_k$) or (EG${}_{k+1}$) fails, and,  by Proposition \ref{th:trivi}, we have  $c_2\le k\le c_1$.

\smallskip\noindent{\bf In the first case}:
\begin{equation}\label{eq:k=a1}
kc_1> k(k-1)+a+(n-k-1)c_2.
\end{equation}
Since $a\ge c_2$, we obtain
\begin{equation}\label{eq:k=a2}
k c_1> k(k-1)+c_2+(n-k-1)c_2.
\end{equation}
Rearranging, we obtain
\begin{equation}\label{eq:k=a3}
0> k^2-(c_2+c_1+1)k+ nc_2.
\end{equation}	
If we have two roots, then the discriminant of \eqref{eq:k=a3} should be positive:
\begin{equation}\label{eq:discriminant}
(c_1+c_2+1)^2-4nc_4>0
\end{equation}
Reordering \eqref{eq:ineq-JSMnewSL} we obtain
\begin{multline*}
0\ge (c_1-c_2+1)^2  - 4c_2(n-c_1-1)=  c_1^2+c_2^2+1-2c_1c_2+2c_1-2c_2 - 4nc_2-4c_1c_2-4c_2=\\
c_1^2+c_2^2+2c_1c_2 +2c_1 +2c_2 +1-4nc_2 = (c_1+c_2+1)^2 - 4c_2n,
\end{multline*}
which contradicts  \eqref{eq:discriminant}

\smallskip\noindent{\bf In the second case}:
\begin{equation}\label{eq:ap=k1}
kc_1+a> k(k+1)+(n-k-1)c_2.
\end{equation}
Since $a\le c_1$, we obtain
\begin{equation}\label{eq:ap=k2}
(k+1) c_1> k(k+1)+(n-k-2)c_2.
\end{equation}
Let $\ell = k+1$. We obtain 	
\begin{equation}\label{eq:ap=k3}
\ell c_1> \ell(\ell-1)+(n-\ell-1)c_2.
\end{equation}
Rearranging, we obtain
\begin{equation}\label{eq:ap=k4}
0> \ell^2-(c_2+c_1+1)\ell+ nc_2.
\end{equation}	
But \eqref{eq:ap=k4} is just \eqref{eq:k=a3}, which leads contradiction again. 	
\end{proof}

\subsection{Sequences defined by extremal degrees and degree sums}

Concerning (P2) we could prove the following theorem which is stronger than (P2${}^*$), where the restriction concerning $c_1$ was   $3 \le c_1 \le \frac{1}{3}\sqrt{\Sigma}.$
\begin{definition}\label{df:GSepsilon}
For any real $\epsilon>0$ define the property ${\varphi}_\epsilon$ as follows:
\begin{equation}\label{eq:GSepsilon}
{\varphi}_\epsilon(n,\Sigma,c_1,c_2)\quad \equiv\quad  2\le c_2 \ \text{ and }\ 3 \le c_1 \le \sqrt{(1-\epsilon) \Sigma}.
\end{equation}
\end{definition}

\begin{theorem} \label{th:GS} For each positive $\epsilon$, the simple region  $\mD[{\varphi}_\epsilon]$ is almost \FG.
\end{theorem}
\begin{proof}
Assume that $D\in \mD[{\varphi}_\epsilon]$ is not graphic. Then $D\in \dds n\Sigma{c_1}{c_2}\subset \mD$  for some parameters $n,\Sigma, c_1,c_2$ with ${\varphi}_\epsilon(n,\Sigma,c_1,c_2)$, and by Theorem \ref{tm:SIMPLE}(2),  $A=\lb(n,\Sigma,c_1,c_2)$ is not \eg.

Then sequence $A$ has the form $(c_1)_ka(c_2)_{n-1-k}$, where $c_1\ge a\ge  c_2$. Applying Theorem \ref{tm:TV+} for $A$ we obtain that
either (EG${}_k$) or (EG${}_{k+1}$) fails.

\smallskip\noindent{\bf In the first case}, when (EG${}_k$) fails, we obtain
\begin{equation}\label{eq:gs1}
kc_1> k(k-1)+a+(n-k-1)c_2.
\end{equation}
Since $a+(n-k-1)c_2=\Sigma-kc_1$, we have
\begin{equation}\label{eq:gs1b}
	kc_1> k(k-1)+\Sigma-kc_1.
\end{equation}
Rearranging, we obtain
\begin{equation}\label{eq:gs2}
0> k^2-(2c_1+1)k+ \Sigma,
\end{equation}	
and so, using $c_1\le \sqrt{(1-\epsilon) \Sigma}$, we have
\begin{equation}\label{eq:gs3}
	0> k^2-(2\sqrt{(1-\epsilon) \Sigma}+1)k+ \Sigma.
\end{equation}	
Thus, the discriminant of \eqref{eq:gs3} should be positive:
\begin{equation}\label{eq:gs4}
		\left ( 2 \sqrt{(1-\epsilon)\Sigma}+1\right) ^2 -4\Sigma>0.
\end{equation}

\smallskip\noindent{\bf In the second case}, when (EG${}_{k+1}$) fails, we obtain
\begin{equation}\label{eq:gs5}
kc_1+a> k(k+1)+(n-k-1)			
\end{equation}
Using $\Sigma-(kc_1+a)=(n-k-1)c_2$, we obtain
\begin{equation}\label{eq:gs5b}
kc_1+a>k(k+1) +(\Sigma-(kc_1+a))
\end{equation}
Rearranging, we obtain
\begin{equation}\label{eq:gs6}
0> k^2-(2c_1-1)k-2a +\Sigma\ge k^2-(2c_1+1)k +\Sigma.
\end{equation}	
and so
\begin{equation}\label{eq:gs7}
0>  k^2-(2\sqrt{(1-\epsilon)\Sigma}+1)k +\Sigma.
\end{equation}	
Thus, the discriminant of \eqref{eq:gs7} should be positive:
\begin{equation}\label{eq:gs8}
\left ( 2 \sqrt{(1-\epsilon)\Sigma}+1\right) ^2 -4\Sigma>0.
\end{equation}
So far we established  the following statement:
\begin{equation}\label{eq:main-gs}
\text{If $D\in\dds n\Sigma{c_1}{c_2}\subset \mD[{\varphi}_\epsilon]$ is not graphic, then }	\left ( 2 \sqrt{(1-\epsilon)\Sigma}+1\right) ^2 -4\Sigma>0.
\end{equation}
But $\left( 2 \sqrt{(1-\epsilon)\Sigma}+1\right) ^2 -4\Sigma>0$ if and only if
$\Sigma<  \frac{1}{4(1-\sqrt{1-\epsilon})^2}$. Since $\Sigma \ge 2n$, we obtain
\begin{equation}\label{eq:main-gs2}
\text{If $D\in\dds n\Sigma{c_1}{c_2}\subset \mD[{\varphi}_\epsilon]$ is not graphic, then $n<  \frac{1}{8(1-\sqrt{1-\epsilon})^2}$}	.
\end{equation}
So we proved there   is a natural number $n_\epsilon$ such that  every element  of $\mD$    with length at least $n_\epsilon$ is graphic. This  completes the proof.
\end{proof}

\begin{remark}
Using \eqref{eq:main-gs2} one could estimate the value $n_\epsilon$ for concrete real numbers. For example, in the case of (P2${}^*$) we have $1-\epsilon=1/9$ and then $n_{8/9}$ is 1.
\end{remark}
\bigskip\noindent Next we prove the statement of (P3${}^*$).
\begin{theorem} \label{th:JMS2}
The simple  degree sequence region $\mD[\varphi^*_{J\!M\!S}]$ is \FG.
\end{theorem}
\begin{proof}
Assume that $D\in \mD[\jms]$, where
$$    {\jms} \equiv (\Sigma-nc_2) (nc_1-\Sigma)\le (c_1-c_2)\left\{(\Sigma-nc_2)(n-c_1-1)+ (nc_1-\Sigma)c_2\right\}.$$
 We need to prove that $D$ is \eg. Fix parameters $n,\Sigma, c_1,c_2$ with $D\in \dds n\Sigma{c_1}{c_2}\subset   \mD[\jms]$.
By Theorem \ref{tm:SIMPLE}(2), it is enough to prove  that the sequence $A=\lb(n,\Sigma,c_1,c_2)$ is  \eg.

The sequence $A$ has the form $(c_1)_ka(c_2)_{n-1-k}$, where $c_1\ge  a\ge  c_2$. Applying Theorem \ref{tm:TV+} for $A$ we obtain that $A$ is \eg\ if and only if $(EG_{k})$ and $(EG_{k+1})$ holds.
Since we can assume $c_2\le k\le c_1$ by Proposition \ref{th:trivi}, $(EG_{k})$ and $(EG_{k+1})$
have the following form:
\begin{equation}\label{eq:eg-k}\tag{$EG_k$}
kc_1\le k(k-1)+(n-k)c_2+(\min\{a,k\}-c_2),
\end{equation}
and
\begin{equation}\label{eq:eg-kplus}\tag{$EG_{k+1}$}
kc_1+a\le (k+1)k+(n-k-1)c_2,
\end{equation}
which can be rearranged as
\begin{equation}\label{eq:eg-kplus2}\tag{$EG^*_{k+1}$}
kc_1\le k(k-1)+(n-k)c_k+(2k-c_2-a).
\end{equation}

\medskip\noindent
Consider the $\jms$ inequality. Using the notation  $b=a-c_2$, we have
\begin{displaymath}
\Sigma=kc_1+a+ (n-k-1)c_2=nc_2+k(c_1-c_2)+b.
\end{displaymath}
So, taking $x=\frac{b}{c_1-c_2}$,   the  LHS of $\jms$ can be written as follows:
\begin{multline}\label{eq:e11}
LHS= \big[k(c_1-c_2)+b\big] \big[(n-k)(c_1-c_2)-b\big ]=\\ (nk-k^2)(c_1-c_2)^2+(n-2k)b(c_1-c_2)-b^2=\\
(c_1-c_2)^2\big[nk-k^2+(n-2k)x-x^2\big].
\end{multline}
Now consider the RHS of $\jms$:
\begin{multline}\label{eq:new1}
RHS= (c_1-c_2)\Big\{\big [(k(c_1-c_2)+b\big ](n-c_1-1)+\big[(n-k) (c_1-c_2)-b\big]c_2\Big\}=\\(c_1-c_2)^2\big[(k+x)(n-c_1-1)+(n-k-x)c_2\big].
\end{multline}
Since $c_1-c_2>0$, putting together \eqref{eq:e11} and  \eqref{eq:new1} we obtain
\begin{displaymath}
nk-k^2+(n-2k)x-x^2 \le (k+x)(n-c_1-1)+(n-k-x)c_2,\\
\end{displaymath}
which can be rearranged as:
\begin{equation}\label{eq:new2}
kc_1 \le k(k-1)+(n-k)c_2+x(2k+x-c_1-1-c_2).
\end{equation}

To derive   \eqref{eq:eg-k} we can assume that $c_2\le k\le c_1$ or $(EG_k)$ holds as we proved it in Proposition   \ref{th:trivi}. Moreover, it is enough to show that the  RHS of \eqref{eq:new2}  is less than, or equal to  RHS of \eqref{eq:eg-k}, that is,
\begin{equation}\label{eq:eg-kxx}
x(2k+x-c_1-1-c_2) \le \min\{a,k\}-c_2.
\end{equation}
Since $0\le {x}< 1$, we have
\begin{equation}\label{eq:harmadikx}
{x}\big(2k+{x}-(c_1+c_2)-1\big)\le 2k-(c_1+c_2)=(k-c_1)+(k-c_2)\le k-c_2.
\end{equation}	
Since $x=\frac{b}{c_1-c_2}\le 1$ and $k\le c_1$, we have
\begin{multline}\label{eq:eg-kx}
x(2k+x-c_1-1-c_2)=b\frac{2k+x-c_1-1-c_2}{c_1-c_2}=\\
b\frac{(c_1-c_2)+(2k-2c_1)+(x-1)}{c_1-c_2}\le b\frac{c_1-c_2}{c_1-c_2}=b.
\end{multline}
Since $\min(a,k)-c_2=\min(b,k-c_2)$, putting together \eqref{eq:eg-kx} and \eqref{eq:harmadikx} we obtain  \eqref{eq:eg-kxx}, which implies that $(EG_k)$ holds.

\bigskip\noindent
To derive   \eqref{eq:eg-kplus} we can assume that $c_2\le k+1\le c_1$ by Proposition   \ref{th:trivi}, and it is enough to show that RHS of \eqref{eq:eg-kplus2} is greater than, or equal to  RHS of \eqref{eq:new2}:
\begin{equation}\label{eq:rhs2}
x(2k+x-c_1-1-c_2)\le 2k-c_2-a.
\end{equation}
But $0\le x\le 1$, so
\begin{equation}\label{eq:rhs2x}
x(2k+x-c_1-1-c_2)=x(2k-c_1-c_2+(x-1))\le (2k -c_1-c_2)
\end{equation}
so \eqref{eq:rhs2x} holds which implies $(EG^*_{k+1})$, and so $(EG_{k+1})$ as well.
This finishes the proof that the region $\mD[\jms]$ is fully graphic.
\end{proof}

\medskip\noindent Gao and Greenhill proved in \cite{GG20} that for any given parameter $\gamma>2$ the infinite set of scale free degree sequences with the given parameter is $P$-stable. However this set is clearly not a degree sequence region. However we believe that this set can be embeded into a $P$-stable simple degree region.

\section{$P$-stable degree sequences}\label{sec:stable}
In this section we are considering \FG, very simple degree sequence regions, and want to prove that they are also $P$-stable  regions. Along this process we will strengthen the Jerrum, McKay and Sinclair's theorem (P1).

\subsection{Every fully graphic very simple degree sequence region  is $P$-stable}\label{sec:strengthJMS}
In Section \ref{sec:intro} the statement (P1) quoted the seminal result of Jerrum, McKay and Sinclair from 1992:
\begin{theorem}[{\cite[Theorem 8.1]{JMS92}}]\label{tm:jms}
The very simple region
\begin{equation}\label{eq:jms}
\mD=\mD[{\varphi}_{\ds \! {J\!M\!S}}]:=\bigcup\left\{\dd{n}{c_1}{c_2}: (c_1-c_2+1)^2\le 4 c_2 (n-c_1-1)\right\}
\end{equation}
is $P$-stable.
\end{theorem}
\noindent Since in Theorem \ref{th:JMS} we proved that the very simple region $\mD[{\varphi}_{\ds \! {J\!M\!S}}]$ is fully graphic, therefore the next statement  is a powerful strengthening of Theorem \ref{tm:jms}:

\begin{theorem}\label{tm:GR2PIntr}
The largest fully graphic very simple region  	
\begin{displaymath}
\mD_{\max}\coloneqq\bigcup\{\dd{n}{c_2}{c_1}:\text{ $\dd{n}{c_2}{c_1}$ is fully graphic}\}
\end{displaymath}
 is  $P$-stable, and so the switch Markov chain is rapidly mixing on $\mD_{\max}$.
\end{theorem}
	
\medskip\noindent
Careful study of the proof of \cite[Theorem 8.1]{JMS92} reveals, that Jerrum, McKay and Sinclair actually proved the following, slightly stronger result.
\begin{theorem}\label{tm:jms2}
The very simple region $\mD[\jmsw]$ is P-stable, where
the property $\jmsw$ is defined as follows:
\begin{equation}\label{eq:bitweaker}
\jmsw\ \equiv \	\forall k\in [1,n] \quad  c_1\cdot k \le k\cdot (k-1)+c_2\cdot(n-k).
\end{equation}
\end{theorem}

Unfortunately, Theorem \ref{tm:jms2} does not yield  Theorem \ref{tm:GR2PIntr} because  the assumption that ``{\em every element of $\dd{n}{c_2}{c_1}$ is graphic}'' does not imply \eqref{eq:bitweaker}. Fortunately, as we already proved (see Lemma \ref{th:surprise}), that the  following, slightly weaker inequality holds for a fully graphic $\dd{n}{c_2}{c_1}$ :
\begin{displaymath}
{\varphi}_{\ds \! {FG}}(n,c_2,c_1) \ \equiv \ \forall k\in [1,n] \quad  c_1\cdot k \le k\cdot (k-1)+c_2\cdot(n-k)+1.	
\end{displaymath}
Using this observation, the following result, which is a direct strengthening of Theorem \ref{tm:jms2},  already  yields Theorem \ref{tm:GR2PIntr}.

\begin{theorem}\label{tm:jms-plusss}
If  a graphic degree sequence	 ${D}=(d_1, d_2, \dots, d_n)$ satisfies
\begin{equation}\label{eq:weaker-revisitedd}
\forall k\in [1,n] \quad  \sum_{i=1}^k d_i\le k\cdot (k-1)+d_n\cdot(n-k)+1,
\end{equation}
then  $\pp{{D}}\le 3\cdot n^{9}$. 	
\end{theorem}

\begin{proof}
Given a graph $G=\langle V,E\rangle$, an {\em alternating trail  of length $\ell$} is a sequence of vertices $v_0,\dots, v_\ell$ such that $\{v_iv_{i+1}\}$ is edge if and only if $i$ is
even, and the pairs $\{\{v_i,v_{i+1}\}:i<\ell\}$ are pairwise distinct. An alternating trail is an {\em alternating path} if the vertices $x_0,\dots, x_n$ are pairwise different apart from the pair $\{x_0,x_1\}$.

The proof of Theorem \ref{tm:jms-plusss} is based on the following Lemma.
\begin{lemma}\label{lm:path-rev}
Let $D^*=D+ 1^{+p}_{+q}$ for some $1\le p,q\le n$, where $D$ satisfies inequality $($\ref{eq:weaker-revi}$)$.
If 	$G$ is a graph with vertex set $V=\{v_1,\dots v_n\}$ and  with degree sequence $D^*$, moreover  $\Gamma(v_p)=\Gamma(v_q)$, then there exists an  alternating trail of odd length 1, 3, 5 or 7 between $v_p$ and $v_q$, which contains one more edges than non-edges.
\end{lemma}

\begin{proof}[Proof of Lemma \ref{lm:path-rev}]
Suppose for the contrary, that there is no such alternating trail. We will describe the structure of $G.$  First, observe that either $v_p=v_q$ or the edge $v_pv_q$ is missing (otherwise there is an alternating trail  of length $1$). Let
\begin{equation}\label{eq:Y-rev}
S=\{v_p,v_q\},\quad X=\Gamma(v_p),\quad
Y=\{y\in V:|X\setminus \Gamma(y)|\ge 2\},\quad
Z=\Gamma(Y)\setminus X.
\end{equation}
\noindent
Observe the following facts:
\begin{enumerate}[{\rm (i)}]
\item \label{Xcl-rev} The set $X$ is a clique (otherwise there is an alternating trail of length 3, namely $v_p\to X\to X \to v_q$).
\item The set  $Y$ is an independent set.

    Indeed, if $\{y_0,y_1\}\in {[Y]}^{2}\cap E$, then $|X\setminus \Gamma(y_i)|\ge 2$ for $i<2$ implies that there is $\{x_0,x_1\}\in {[X]}^{2}$ such that $\{x_i,y_i\}\notin E$ for $i<2$. Thus $v_p,x_0,y_0,y_1,x_1,v_q$ is an alternating trail  of length 5.
\item \label{Zcl-rev}The set $Z$ is a clique..

    Indeed, if $\{z_0,z_1\}\in {[X]}^{2}\setminus  E$, then there are $y_0,y_1\in Y$ such that $\{z_i,y_i\}\in E$ for $i<2$, but we can not guarantee that $y_0\ne y_1$. Since $|X\setminus \Gamma(y_i)|\ge 2$ for $i<2$, there is $\{x_0,x_1\}\in {[X]}^{2}$ such that $\{x_i,y_i\}\notin E$ for $i<2$. Thus $v_p,x_0,y_0,z_0,z_1,y_1,x_1,v_q$ is an alternating trail of length 7, but not necessarily a path.

\item \label {XZ-rev} The induced bipartite graph $G[X,Z]$ is complete.

    Indeed, if $x\in X$ and $z\in Z$ with $\{x,z\}\notin E$, then pick first $y\in Y$ with $\{y,z\}\in E$. Since $|X\setminus \Gamma(y)|\ge 2$, we can pick $x'\in X$ such that $x'\ne x$ and $\{x',y\}\notin E$. Then $x_p,x,z,y,x',v_q$ is an alternating trail of length 5.
\item The sets $\{v_p,v_q\}, X, Y, Z$ are pairwise disjoint.
\end{enumerate}
Let $$K= X \cup Z\quad\text{ and }\quad R = V\setminus (K \cup Y\cup S).$$
We have $|K| +|Y| + |R|+|S|=|V|=n$. Putting together \eqref{Xcl-rev}, \eqref{Zcl-rev} and  \eqref{XZ-rev}, we obtain
\begin{enumerate}[(i)]\addtocounter{enumi}{5}
\item $K$ is a clique.
\end{enumerate}

\noindent Write $k=|K|$. We will estimate  the sum of the degrees of the vertices in $K$. To start with, write
\begin{multline}\label{eq-rev:main1}
\sum_{i=1}^kd_i \ge  \sum_{v\in  K}\deg(v)=\\  \sum_{v\in  K}|\Gamma(v)\cap K|+
\sum_{v\in  S}|\Gamma(v)\cap K|+\sum_{v\in R}|\Gamma(v)\cap K|+\sum_{y\in Y}|\Gamma(y)\cap K|
\end{multline}
Since $K$ is a clique,
\begin{equation}\label{eq-rev:K-clique}
\sum_{v\in  K}|\Gamma(v)\cap K|= k\cdot (k-1).
\end{equation}
Since $\Gamma(v)=X\subset K$ for $v\in S$, we have
\begin{equation}\label{eq-rev:S}
\sum_{v\in  S}|\Gamma(v)\cap K|=|S||X|.
\end{equation}
Since $|X\setminus \Gamma(v)|\le 1$ for $v\in R$,
we have
\begin{equation}\label{eq-rev:R}
\sum_{v\in  R}|\Gamma(v)\cap K|\ge |R|(|X|-1).
\end{equation}
By the construction, $\Gamma(Y)\subset K$, so
\begin{equation}\label{eq-rev:Y}
\sum_{y\in  Y}|\Gamma(y)\cap K|\ge |Y|\cdot d_n.
\end{equation}
Putting together, we have
\begin{equation}\label{eq-rev:444}
\sum_{i=1}^kd_i\ge k\cdot(k-1)+ |S|\cdot|X|+|R|\cdot (|X|-1)+|Y|\cdot d_n.
\end{equation}
Since  $|X|=\deg(v_p)\ge d_n+(3-|S|)$, we obtain
\begin{equation}\label{eq-rev:main2}
\sum_{i=1}^kd_i\ge  k\cdot(k-1)+ |S|\cdot (d_n+3-|S|)+|R|(d_n+2-|S|)+|Y|d_n.
\end{equation}
Observe that $|R|+|Y|+|S|=n-k$. Clearly $|S|=1$ or $|S|=2$.

\medskip
\noindent 	If $|S|=1$, then \vspace{-5pt}
\begin{multline}\label{eq-rev:S1}
|S|\cdot (d_n+3-|S|)+|R|(d_n+2-|S|)+|Y|d_n=d_n+2+ (|R|+|Y|)d_n+|R|=\\
(|R|+|Y|+|S|)d_n +2+|R|= (n-k)d_n+2.
\end{multline}
If $|S|=2$, then \vspace{-5pt}
\begin{multline}\label{eq-rev:S2}
|S|\cdot (d_n+3-|S|)+|R|(d_n+2-|S|)+|Y|d_n=2(d_n+1)+ (|R|+|Y|)d_n=\\
(|R|+|Y|+|S|)d_n +2= (n-k)d_n+2.
\end{multline}
So in both case, from  \eqref{eq-rev:main2}  we obtain
\begin{equation}\label{eq-rev:main3}
\sum_{i=1}^k d_i\ge k\cdot(k-1)+ (n-k)d_n+2,
\end{equation}
which contradicts \eqref{eq:weaker-revi}.  So we proved  Lemma \ref{lm:path-rev}.
\end{proof}

\noindent The proof of   Theorem \ref{tm:jms-plusss} from Lemma \ref{lm:path-rev} is similar to the proof of \cite[Theorem 8.1]{JMS92} from \cite[Lemma 1]{JMS92}.

\medskip\noindent
Assume that $G'$  is a graph such that the degree sequence of $G'$ is $D'=D+1^{+i}_{+j}$ for some $1\le i<j\le n$.

If $\Gamma_{G'}(v_i)=\Gamma_{G'}(v_j)$, then  we can apply Lemma \ref{lm:path-rev} for $G=G'$, $p=i$ and $q=j$ to obtain  an  alternating trail  $P$ of odd length 1, 3, 5 or 7 between $v_i$ and $v_j$, which contains one more edges than non-edges. Flipping edges and non-edges along the trail $P$ we obtain a graph $G^\dag$ which is a realization of $D.$

If $\Gamma_{G'}(v_i)\ne \Gamma_{G'}(v_i)$, then  there is an alternating trail $Q$ of length $2$ between $v_i$ and $v_j$. Assume that $Q=v_iv_mv_j$, where $v_iv_m$ is an edge, and $v_mv_j$ is a non-edge. Flipping edges along  trail $Q$ we obtain a graph $G^*$ with degree sequence  $$D^*=D'+1^{+j}_{-i}=D+1^{+i}_{+j}+1^{-i}_{+j}=D+1^{+j}_{+j}.$$

Now we can apply Lemma \ref{lm:path-rev} for $G=G^*$ with $p=q=j$ to obtain  an alternating trail  $P$ of odd length 1, 3, 5 or 7 from $v_j$ to $v_j$, which contains one more edges than non-edges. Flipping edges and non-edges along the trail $P$ we obtain a graph $G^\dagger$ which is a realization of $D.$

How much information  should  we use to obtain back $G'$ from $G^\dagger$? We need to know if we were in case  $\Gamma_{G''}(v_i)=\Gamma_{G''}(v_j)$ or in case $\Gamma_{G''}(v_i)\ne \Gamma_{G''}(v_j)$.

If $\Gamma_{G''}(v_i)=\Gamma_{G''}(v_j)$, we should know $P$. The  trail  $P$ contains at most 8 vertices, so this is at most $n^8$ possibilities.

If $\Gamma_{G''}(v_i)\ne \Gamma_{G''}(v_j)$, we should know $P$ and $Q$. Since the first and the last element of $P$ are the same, we have at most $n^7$ possibilities for $P$. Knowing $P$  we can compute $G^*$, and we also know  $v_i$ or $v_j$. We should know  $Q$. We know one vertex ($v_i$ or $j$) from $Q$. So knowing P  we have
$n^2$ possibilities for  $Q$. Knowing $Q$ we can compute  $G''$. We should know which endpoint  of $Q$ is $v_i$ and which is $v_j$. In this case we have at most $2\cdot n^2\cdot n^7=2\cdot  n^9$ possibilities.

Putting together, for a given $G^*$ we have at most $n^8+2\cdot n^9\le 3\cdot n^9$ possibilities for $G^\dagger$.
\end{proof}

The next theorem gives us a method to prove that a simple region $\mD[{\varphi}]$ is P-stable. Namely, it is enough to prove that $\lb(n,\Sigma,c_1,c_2)$ satisfies \eqref{eq:weaker-revisitedd} from Theorem   \ref{tm:jms-plusss} whenever ${\varphi}(n,\Sigma,c_1,c_2)$ holds.

\begin{theorem}\label{tm:leg-p-stable}
If $n>c_1\ge c_2$ and $nc_1\ge \Sigma\ge nc_2$ are natural numbers,
$\Sigma$ is even,  then the following are equivalent:
\begin{enumerate}[(1)]
\item $\lb(n,\Sigma,c_1,c_2)$ satisfies \eqref{eq:weaker-revisitedd} from Theorem
\ref{tm:jms-plusss},
\item every $D\in \dds{n}{\Sigma}{c_1}{c_2}$ satisfies the \eqref{eq:weaker-revisitedd} from Theorem \ref{tm:jms-plusss}.
\end{enumerate}
\end{theorem}

\begin{proof}
To show that (1) implies (2), let $D=(d_1,\dots,d_n)$ be an arbitrary element of $\dds{n}{\Sigma}{c_1}{c_2}$, and fix  $1\le k\le n$. Write $LEG(n,\Sigma,c_1,c_2)=(e_1,\dots, e_n)$. Then,  $\sum_{i=1}^kd_i\le \sum_{i=1}^ke_i$, and  $\sum_{i=1}^ke_i\le k(k-1)+ (n-k)e_n+1$ because
$\lb(n,\Sigma,c_1,c_2)$ satisfies \eqref{eq:weaker-revisitedd}. Putting together these two inequalities we obtain
\begin{displaymath}
\sum_{i=1}^kd_i\le  k(k-1)+ (n-k)e_n+1=k(k-1)+ (n-k)c_2\le k(k-1)+ (n-k)d_n+1,
\end{displaymath}
which implies that  \eqref{eq:weaker-revisitedd} holds for $D$ and $k$.
\end{proof}

\bigskip\noindent
By $(P2)$, the simple region $\mathbb D[{\varphi}_{\ds \! {GS}}]=\mD[{\varepsilon}_{8/9}]$
is P-stable. We also proved that $\mD[{\varphi}_{\varepsilon}]$ is fully graphic
for ${\varepsilon}>0$. The next question is very natural.
\begin{problem}
Is the simple region $\mD[{\varphi_\epsilon}]$  P-stable for ${\varepsilon}>0$?
\end{problem}

\bigskip
\subsection{Construction of  $P$-stable families with special properties}\label{sec:another}
\medskip

We  demonstrate the existence of a fully graphic  simple region,
whose  P-stability can be derived from   Theorem \ref{tm:jms-plusss}, whereas
application of  \cite[Theorem 8.3]{JMS92} does not yield its P-stability.

\begin{theorem}\label{tm:d_0}
The simple region
\begin{equation}\label{eq:simple_d0}
\mD_0=\bigcup\{\dds{2m}{4m}{m}{1}:m\ge 4\}
\end{equation}
is fully graphic and  P-stable, although $\mD_0\cap \mD[\varphi^*_{J\!M\!S}]=\emptyset$.
\end{theorem}

\begin{proof}
First, observe that 	
\begin{equation}\label{eq:D_m}
D_m\coloneqq(m)_2(3)_1(1)_{2m-3}=\lb(2m,4m,m,1).
\end{equation}

\begin{lemma}\label{lm:D_0_1}
For $m\ge 4$, $D_m$ does not satisfy the $\varphi^*_{J\!M\!S}$ inequality.
\end{lemma}

\begin{proof}
Indeed, $n=2m$, $c_1=m$, $c_2=1$, $\Sigma=4m$, so \vspace{-10pt}
\begin{multline*}
LHS_{\varphi^*_{J\!M\!S}}-RHS_{\varphi^*_{J\!M\!S}}=\\
(\Sigma-c_2n)(c_1n-\Sigma)- (c_1-c_2)[(\Sigma-c_2n)(n-c_1-1)+(c_1n-\Sigma)c_2]=\\
4(m^2 - 2m)m-[2((m - 1)m + m^2 - 2m)(m - 1)]= 2m^2-6m.
\end{multline*}
which is positive for $m\ge 4$, so we proved the Lemma.
\end{proof}
\begin{lemma}\label{lm:dm_0_2}
$D_m$ is graphic.
\end{lemma}
\begin{proof}
By the Tripathi-Vijay Theorem \ref{tm:TV+}, we should check only $EG_2$ and $EG_3$ for $D_m$. But
\begin{equation}\label{eq:dm_0_2_1}\tag{$EG_2$}
\sum_{i=1}^2d_i =2m < 2+(2+(2m-3)\cdot 1) =2(1-2)+\sum_{i=3}^{2m}\min(d_i,2),
\end{equation}	
and
\begin{equation}\label{eq:dm_0_2_2}\tag{$EG_3$}
\sum_{i=1}^3d_i =2m+3 =6+  (2m-3)\cdot 1 =  3(3-1)+\sum_{i=4}^{2m}\min(d_i,3).
\end{equation}	
\end{proof}

\begin{lemma}\label{lm:dm-3}
$D_m$ satisfies \eqref{eq:weaker-revisitedd} from Theorem \ref{tm:jms-plusss}.
\end{lemma}
\begin{proof}
If $k=1,2,3$, inequality \eqref{eq:weaker-revisitedd} is the following:
\begin{gather}
d_1=m \le 1(1-0)+(2m-1)1+1 \tag{$k=1$}\\
d_1+d_2=2m < 2(2-1)+(2m-2)1+1, \tag{$k=2$}\\
d_1+d_2+d_3=2m+3\le 3(3-1)+(2m-3)1+1, \tag{$k=3$}\\
\end{gather}
If $k\ge 3$, then $EG_k$ implies $EG_{k+1}$, because , the LHS is increased by 1, and the RHS is increased by $2k-1$.
\end{proof}
The lemmas together prove the theorem.
\end{proof}

\section{Large  \FG\ regions in very simple P-stable  regions} \label{sec:P-graphic}

In the first two  subsections we review the necessary facts about split graphs and  Tyshkevich product.

\subsection*{Split graphs.}
A $G=(V,E)$ graph is a \textbf{split graph} if its vertices can be partitioned into a clique and an independent set. Split graphs were introduced by Földes and Hammer (\cite{FH77}).

Split graphs are recognizable from their degree sequences:
\begin{theorem}[Hammer and Simeone, 1981 \cite{HS}, Tyshkevich, Melnikov and Kotov \cite{T81}]\label{th:HS}
Assume that $G$ is a graph with degree sequence $D=(d_1,\dots, d_n)$, where $d_1\ge d_2\ge\dots\ge d_n$. Let $m$ be the largest value of $\, i$, such that  $d_i\ge i-1.$ Then $G$ is a split graph if and only if
$$
\sum_{i=1}^m d_i = m(m-1) + \sum_{i=m+1}^n d_i.
$$
\end{theorem}
\begin{remark}\label{th:name}
Consequently,  if one realization of a   degree sequence $D$  is  a split graph, then all realizations of $D$ are split graphs as well. Such a degree sequence is referred  as
\textbf{\em split degree sequence}.
\end{remark}

We will write $G=((U,W),E)$ to mean that $G$ is a split graph with vertex set $U\cup W$, $U$ is a clique and $W$ is an independent set. Let us remark that $U$ and $W$ are not necessarily unique.

\begin{theorem}\label{tm:split-in-non-graphic}
If $\dd{n}{c_1}{c_2}$ is not \FG,  then $\dd{n}{c_1}{c_2}$ contains a split degree sequence.
\end{theorem}

\begin{proof}
Fix a non-graphic $D\in \dd{n}{c_1}{c_2}$. Write $\Sigma=\sum D$. By Corollary \ref{th:check-simple}  the sequence $\lb(n,\Sigma,c_1,c_2)$ is not graphic.

\begin{lemma}\label{lm:biner}
There is $c_2\le \ell\le c_1$ such that 	
\begin{equation}\label{nessfulleq:wit}
\ell c_1>\ell(\ell-1) + (n-\ell)c_2.
\end{equation}
\end{lemma}
\begin{proof}[Proof of the Lemma]
Write $\lb(n,\Sigma,c_1,c_2)=(d_1,\dots, d_n)=((c_1)_k,a, (c_2)_{n-2})$, where $c_2\le a\le c_1$.

By the Tripathi-Vijay Theorem \ref{tm:TV+}, either EG${}_k$ or EG${}_{k+1}$ fails.
By Proposition \ref{th:trivi} property (EG${}_\ell$) holds for each $\ell\le c_2$ or $\ell> c_1$. So we can assume that $c_2\le k<c_1$.

\noindent{\bf Case 1}: If (EG${}_k$) fails, then
\begin{equation}\label{eq:witness1}
\sum_{i=1}^kd_i> k(k-1)+\sum_{i=k+1}^n\min(d_i,k),
\end{equation}
and so
\begin{equation}\label{eq:witness12}
kc_1> k(k-1)+ \min(a,k)+ (n-k-1)c_2,
\end{equation}
therefore
\begin{equation}\label{eq:witness13}
kc_1> k(k-1)+  (n-k)c_2.
\end{equation}

\noindent{\bf Case 2}: If (EG${}_{k+1}$) fails, then
\begin{equation}\label{eq:witness2}
\sum_{i=1}^{k+1}d_i> (k+1)k+\sum_{i=k+2}^n\min(d_i,k).
\end{equation}
and so
\begin{equation}\label{eq:witness22}
kc_1+a>(k+1)k+ (n-k-2)c_2,
\end{equation}
therefore
\begin{equation}\label{eq:witness23}
(k+1)c_1>(k+1)k+ (n-k-1)c_2.
\end{equation}
So either $\ell=k$ or $\ell=k+1$ has the following property: $c_2\le \ell\le c_1$ and
\begin{equation}\label{eq:witness-ell}
\ell c_1>\ell(\ell-1)+ (n-\ell)c_2.
\end{equation}
So we proved the Lemma.
\end{proof}
\medskip\noindent Let
\begin{equation}
\sigma=(n-\ell)c_2,\qquad c=\lfloor \sigma/\ell \rfloor,\qquad
{\alpha}=\sigma-\ell c.
\end{equation}
Then
\begin{displaymath}
(n-\ell)c_2={\sigma}={\alpha}(c+1)+(\ell-{\alpha})c.
\end{displaymath}
Consider the following degree sequence
\begin{equation}\label{eq:split-sequence}
D=(\ell+c)_{{\alpha}}(\ell+c-1)_{\ell-{\alpha}}(c_2)_{n-\ell}
\end{equation}
\begin{lemma}\label{lm:real}
The previous degree sequence is $D\in \dd{n}{c_1}{c_2}$, and it is a graphic split sequence.
\end{lemma}

\begin{proof}[Proof]
First observe that $D$ is graphic. Really, it has a realization $G=\langle V,E\rangle$ on the vertex set $V=\{v_j:j<\ell\}\cup\{w_k:k<n-\ell\}$ with
\begin{equation}\label{eq:half}
E=\big\{(v_i,v_j):i<j<\ell\big\}\cup \big\{(v_{i \bmod \ell},w_{i \bmod{(n-\ell)}}):i<\sigma\big\}.
\end{equation}
Next observe that $\ell$ is the largest $j$ such that $d_j\ge j-1$. Moreover,
\begin{multline*}
\sum_{i=1}^\ell d_i= (\ell+c){{\alpha}}+(\ell+c-1)({\ell-{\alpha}})=\sigma+\ell(\ell-1) \\= \ell(\ell-1)+c_2({n-\ell})=\ell(\ell-1)+\sum_{i=\ell+1}^nd_i.
\end{multline*}
Thus the  degree sequence $D$ is a split degree sequence by Theorem \ref{th:HS}.
\end{proof}
\noindent
This  completes the proof of the Theorem \ref{tm:split-in-non-graphic}.
\end{proof}

\subsection*{Tyshkevich product}
\begin{definition}[Tyshkevich \cite{T00}]\label{df:t-product}
Let $G=(\langle U, W\rangle ; E)$ be a split graph and $H=(V,F)$ be an arbitrary graph. We define the  {\em composition} graph $K= G \circ H$ as follows: $K$ consists of a copy of $G$, and a copy of $H$ and of all the possible new edges $(u,v)$  where $u \in U, v \in V.$ More formally,
\begin{displaymath}
V(K)=U\cup W\cup V \text{ and } E(K)=E\cup F\cup\{(u,v):u\in U, v\in V\}.
\end{displaymath}
\end{definition}
\noindent Observe that the first operand in this operation is always a split graph.

\medskip\noindent
Barrus \cite[Theorem 3.5]{EMT}  proved the following (see also \cite[Theorem 6]{barrus16}):

\begin{theorem}[Barrus]\label{tm:t-main}
Assume that  $G=(\langle U, W\rangle ; E)$ is a split graph and $H=(V,F)$ is an arbitrary graph. Let $K= G\circ H$. Then
\begin{equation}\label{eq:t-main`'}
|\mathcal{G}(\mathbf {d}(K))|=|\mathcal{G}(\mathbf {d}(G))|\cdot |\mathcal{G}(\mathbf {d}(H))|.
\end{equation}
\end{theorem}

\bigskip
\subsection*{How to obtain ,,not P-stable'' from ,,not almost \FG''?}
\medskip

\begin{theorem}\label{tm:P-stb2EG}
Assume that
\begin{displaymath}
\mD=\bigcup_{k\in {\mathbb N}}\dd{n_k}{c_{k,1}}{c_{k,2}}\text{ and }
\mD'=\bigcup_{k\in {\mathbb N}}\dd{n'_k}{c'_{k,1}}{c'_{k,2}}
\end{displaymath}
are very simple degree sequence regions.  If $\mD$ is not almost \FG, then $\mD'$ is not P-stable provided:
\begin{enumerate}[{\rm (1)}]
\item $\lim_{k\to \infty}(n'_k-n_k)/\ln(n'_k)=+\infty$,
\item $c'_{k,2}\le c_{k,2}$,
\item $c'_{k,1}\ge c_{k,1}+(n'_k-n_k)$.
\end{enumerate}
\end{theorem}

\begin{proof}
We will use a construction, which is similar to the one Jerrum,  McKay and Sinclair derived in \cite[Lemma 8.1]{JMS92}, and based on the result \cite[Corollary 6.2]{horse}.
\begin{lemma}\label{lm:big-gap}
For each natural number $m\ge 1$ the following sequence
\begin{equation}\label{eq:spec-sequence}
\mathbf h_m=(2m-1,2m-2,\dots,m+1,m,m,m-1,\dots,2,1),
\end{equation}
has exactly one realization $H_m=(V,E)$ on the vertex set $V=\{v_1,\dots, v_{2m}\}$, namely
\begin{displaymath}
E=\{(i,j):m+1\le i<j\le 2m\}\cup\{(i,j): 1\le i\le m<j\le 2m, i+j\le 2m+1\}.
\end{displaymath}
However the sequence
\begin{equation}\label{eq:spec-sequence2}
\mathbf h'_m=\mathbf h_m+ 1^{+2m}_{+m}= ({2m-1},2m-2,\dots,m+1,\mathbf{m+1},m-1,m-2,\dots,2,\mathbf{2})
\end{equation}
has at least
$\Theta (e^{{{\delta}m}})$	realizations for some ${\delta}>0$.
\end{lemma}

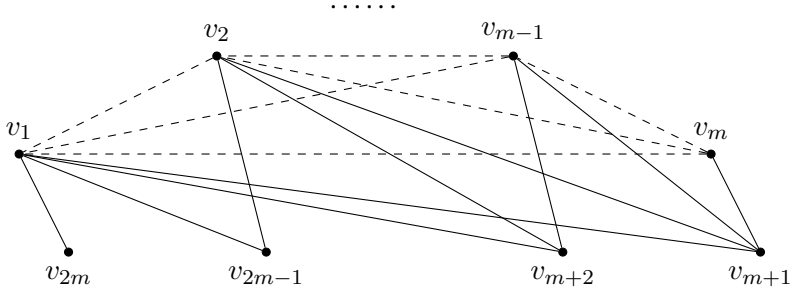
\begin{figure}[H]
	\centering
	\begin{tikzpicture}[scale=1.3]
	\tikzstyle{vertex}=[draw,circle,fill=black,minimum size=3,inner sep=0]
	\node[vertex] (y1) at (0.5,1) [label=south:{$v_{2m}$}] {};
	\node[vertex] (y2) at (2.5,1) [label=south:{$v_{2m-1}$}] {};
	\node[vertex] (y3) at (5.5,1.) [label=south:{$v_{m+2}$}] {};
	\node[vertex] (y4) at (7.5,1) [label=south:{$v_{m+1}$}] {};
	\node[vertex] (x4) at (7,2) [label=north:{$v_{m}$}] {};
	\node[vertex] (x3) at (5,3) [label=north:{$v_{m-1}$}] {};
	\node[vertex] (x2) at (2,3) [label=north:{$v_2$}] {};
	\node[vertex] (x1) at (0,2) [label=north:{$v_1$}] {};
    \node at (3.5,3.5)  {\dots\dots};
	 \draw[dashed] (x2)--(x1);
	 \draw[dashed] (x3)--(x1);
	 \draw[dashed] (x3)--(x2);
	 \draw[dashed] (x4)--(x1);
	 \draw[dashed] (x4)--(x2);
	 \draw[dashed] (x4)--(x3);

	\draw (x1) -- (y4);
	\draw (x1) -- (y3);
	\draw (x1) -- (y2);
	\draw (x1) -- (y1);
	\draw (x2) -- (y3);
	\draw (x2) -- (y2);
	\draw (x2) -- (y4);
	\draw (x3) -- (y3);
	\draw (x3) -- (y4);
	\draw (x4) -- (y4);
	\end{tikzpicture}
\caption{The unique realization  of $\mathbf h_m$.}	
\label{fig:hm}
\end{figure}

\begin{proof}
It is easy to see that the degree sequence $\mathbf h_m $ has exactly one realization (or see, for example, \cite[Lemma 8.1]{JMS92}).

The calculation concerning the number of realizations of  $\mathbf h'_m$  follows from \cite[Corollary 6.2]{horse}, which claims that the bipartite degree sequence
\quad $((m,m-1,\dots, 2, 2), ((m,m-1,\dots, 2, 2))$ \quad has
$$
\Theta \left( {\left(\frac{3+\sqrt{5}}{2}\right)}^{{m}} \right)
$$	
realizations.
\end{proof}
Let $I=\{k\in \mathbb N:\text{$\dd{n_k}{c_{k,1}}{c_{k,2}}$ contains a non-\eg\ sequence}\}$. By the assumption of Theorem, the set $I$ is infinite. Replacing $\mD$ with $\bigcup\{\dd{n_k}{c_{k,1}}{c_{k,2}}:{n\in I}\}$, we can assume that every $\dd{n_k}{c_{k,1}}{c_{k,2}}$ contains a non-\eg\ sequence. Therefore, by Theorem \ref{tm:split-in-non-graphic}, every $\dd{n_k}{c_{k,1}}{c_{k,2}}$ contains a split degree sequence $D_k$. Let $G_k$ be the unique realization of $D_k$. Furthermore, for each $k$ let $H_k^*=H_{\ds{n_k'-n_k}}$ be the unique realization of the graphic degree sequence $\mathbf h_{\ds{n_k'-n_k}}$.

Let $\mathbf e_k$ be the degree sequence of the Tyshkevich product $G_k \circ H_k$. Clearly $\mathbf e_k\in \dd{m_k'}{c'_{k,2}}{c'_{k1}}$ by the construction. Then, by Theorem \ref{tm:t-main}, $\mathbf e_k$ has exactly one realization because both $\mathbf d_k$ and $\mathbf h_{\ds{n_k'-n_k}}$ have exactly one realization.

However, for some $i$ and $j$, the sequence $\mathbf e_k+1^{+i}_{+j}$ has at least
$$
C\cdot  e ^{{{\delta}(n'_k-n_k)}}
$$	
realizations by Theorem \ref{tm:t-main} and by the second  part of the Lemma \ref{lm:big-gap}.

Let $p$ be an arbitrary polynomial. Then
\begin{displaymath}
\lim_{k\to \infty}\ln \Big(\frac{C\cdot  e^{{{\delta}(n'_k-n_k)}}}{p(n'_k)}\Big)\ge
\lim_{k\to \infty}\Big(C'\cdot  ({{n'_k-n_k}})- C''\cdot \ln (n'_k)\Big) =\infty,
\end{displaymath}
by assumption (1) of this Theorem. Thus, the ration of the number of realizations of $\mathbf e_k$ and $p(k)$ tends to infinity. So $\mD'$ is not P-stable.
\end{proof}

\begin{theorem}\label{tm:P2FG}
Assume that $f,g,h\in \mathbb N^{\mathbb N}$ are increasing functions. If the  \cone\ $\mathbb D(f,g)$ is P-stable, then $\mathbb D(f+h,g-h)$ is almost fully graphic provided
\begin{enumerate}[{\rm (1)}]
\item $f(n+k)\le f(n)+k$ for each $n,k\in {\mathbb N}$,
\item $lim_{n\to \infty}h(n)/\ln(n)=\infty$.
\end{enumerate}
\end{theorem}
	
\begin{proof}[Proof of Theorem \ref{tm:P2FG}]
Assume on the contrary that $\mathbb D(f+h,g-h)$ is not almost \FG. Let
\begin{enumerate}[(i)]
\item $n_k=k$, $c_{k,1}=g(n_k)-h(n_k)$, $c_{k,2}=f(n_k)+h(n_k)$,
\item $n_k'=n_k+h(n_k)$,  $c'_{k,1}=g(n'_{k})$  $c'_{k,2}=f(n'_{k})$.
\end{enumerate}
The assumption Theorem \ref{tm:P2FG} (2) implies that \ref{tm:P-stb2EG}(1) holds. The assumption Theorem \ref{tm:P2FG} (1) implies that \ref{tm:P-stb2EG}(2) holds. Finally Theorem \ref{tm:P2FG} (1) also implies that $g(n_{k,1})\ge c_{k,1}+h(n_k)$, and $g(n'_{k,1})\ge g(n_{k,1})$ because $g$ is monotone. So \ref{tm:P-stb2EG}(3) holds.

Hence, we can apply Theorem \ref{tm:P-stb2EG} to obtain that $\mD(f,g)$ is not P-stable.
\end{proof}

\begin{corollary}\label{cor:P2Gr}
Assume that $0\le {\varepsilon}_2<{\varepsilon}'_2<{\varepsilon}'_1< {\varepsilon}_1\le 1$. If the very simple region $\mD\coloneqq\bigcup_{n\in {\mathbb N}}\dd{n} {\lfloor {\varepsilon}_1\cdot n \rfloor }{\lceil {\varepsilon}_2\cdot n\rceil } $ is P-stable, then the region $\mD'=\bigcup_{n\in {\mathbb N}}\dd{n}
{\lfloor {\varepsilon}'_1\cdot n \rfloor }{\lceil {\varepsilon}'_2\cdot n\rceil } $ is almost fully graphic.
\end{corollary}

\section{A common bound of the growing rate in P-stable regions}

In the literature, various P-stable families of degree sequences are described. However, the empirical evidence is that the polynomial $p_0(n)=n^{10}$ has the following property:  if $\mD$ is  P-stable, then $\pp D\le p_0(|D|)$ for all but finitely many  $D\in \mD$. This observation is notable and leads to the following bold conjecture.

\begin{conjecture}
There is a polynomial $p^*(n)$ such that for each  P-stable family $\mD$ of degree sequences  (or, just  for each P-stable simple region $\mD$)
$$\pp D\le p^*(|D|)$$
for all but finitely many graphic $D\in \mD$.
\end{conjecture}

\bibliographystyle{plain}

\section*{Appendix: the different definitions of P-stability are really equivalent}\label{sec:app}

Given a degree sequence $D$ of length $n$, define the following families of
degree sequences:
\begin{align*}
\DD{-+} &=\big\{D+1^{-i}_{+j}: 1\le i\ne j\le n \big\},\medskip\\
\DD{++} &=\big\{D+1^{+i}_{+j}: 1\le i\ne j\le n \big\},\medskip\\
\DD{+2} &=\big\{D+1^{+i}_{+i}: 1\le i\le n \big\}.
\end{align*}
The families $\DD{--} $ and $\DD{-2} $ are defined analogously.

In 1989 Jerrum and Sinclair introduced the so-called \emph{Jerrum-Sinclair Markov Chain} (JSMC) in their seminal paper \cite{JS89} on the approximation of the zero-one permanents. Later they used the same Markov chain to sample certain graph realization classes on labelled vertices in  \cite{JS90}. They introduced there the notion of P-stability: a family $\mathcal D$  of degree sequences  is \textbf{P-stable } if and only if there is a polynomial $p(n)$ such that for each graphic  sequence $D\in \mathcal D$ with length $n$ we have
\begin{displaymath}\tag{JS}
|\mathbb G(\DD{--})\cup \mathbb G(\DD{-2})|\le p(n)\cdot |\mathcal  G(D)|.
\end{displaymath}
In 1992 Jerrum, McKay and Sinclair gave more results about P-stable degree sequences (\cite[Subsection 8.2]{JMS92}). However they used there a different definition. They say that a family $\mathcal D$  of degree sequences  is {\em P-stable } if and only if there is a polynomial $p_1(n)$ such that for each graphic  sequence $D\in \mathcal D$ with length $n$ we have
\begin{displaymath}\tag{JMS}
|\mathbb G(\DD{-+})|\le p_1(n)\cdot |\mathcal  G(D)|.
\end{displaymath}
The authors made the remark (without proof) that, while the two definitions formally are different, they are equivalent.

There has been studied another Markov chain based approach to sample graph realizations for at least three decades, the chain is called \emph{switch Markov Chain}. In 2022, the rapid mixing of the switch Markov chain was proven on P-stable degree sequences, encompassing all simple, bipartite, and directed degree sequences (\cite{P-stable}), providing the currently available strongest result. However, that  paper presented a third definition for P-stability \cite[Definition 1.2]{P-stable}. They say that a family $\mathcal D$  of degree sequences  is {\em P-stable } if and only if there is a polynomial $p_2(n)$ such that for each graphic sequence  $D\in \mathcal D$ with length $n$ we have
\begin{displaymath}\tag{EGMMSS}
|\mathbb G(\DD{++})|\le p_2(n)\cdot |\mathcal  G(D)|.
\end{displaymath}
The paper states (again, without  proof) that this definition is equivalent to the former ones.

The following theorem yields immediately that the three definitions of P-stability  are indeed equivalent.

\begin{theorem}\label{tm:appendix}
Assume that $D=(d_1,\dots, d_n)$ is a graphic degree sequence.
\begin{enumerate}[\rm (a)]
\item \label{gdpp} $\max\big(|\mathbb G(\DD {++})|,|\mathbb G(\DD {--})|\big)\le n^2\cdot\big( |\mathbb G(\DD {+-})|+|\mathcal G(D)|\big), $
\item $\max\big(|\mathbb G(\DD {+2})|,|\mathbb G(\DD {-2})|\big )\le n^2\cdot |\mathbb G(\DD {+-})|$,
\item $|\mathbb G(\DD {+-})|\le (n^4+n^2)\cdot \min\big(|\mathbb G(\DD {++})|,|\mathbb G(\DD {--})|\big)$.
\end{enumerate}
\end{theorem}

\noindent
\begin{remark*}The proofs for (a) and (b) are straightforward. However, proving (c) presents a greater challenge.
\end{remark*}

\begin{proof}
We assume that the vertex set of the realizations is $\{v_1,\dots, v_n\}$. \medskip

\medskip
\noindent (a)
Assume that $G\in  \mathcal G(D+1^{+i}_{+j}) \subset  \mathbb G(\DD {++})$ for some $1\le i\ne j\le n$. We will define a graph $G'\in \mathbb G(\DD {+-})\cup \mathcal G(D)$ such that  the symmetric difference of $E(G)$ and $E(G')$ is one edge.

If $(v_i,v_j)\in E(G)$, then let $G'=G-(v_i,v_j)$. Then
$\mathbf d(G')=\mathbf d(G)+1^{-i}_{-j}=D$, so $G'$ satisfies the requirements.

Assume that $(v_i,v_j)\notin E(G)$. Since $\deg_G(v_i)=d_i+1\ge 1$, there is a $k$
such that $v_iv_k$ is an edge in $G$. Since $(v_i,v_j)$ is a non-edge,  $k\ne j$.
Let $G'=G-v_iv_k$. Then $\mathbf d(G')=\mathbf  d(G)+1^{-i}_{-k}= D+1^{+i}_{+j}+1^{-i}_{-k}=D^{+j}_{-k}$, so $G'$ satisfies the requirements.

From $G$ you can get back $G'$ provided you know the  symmetric difference of $E(G)$ and $E(G')$, which is just one pair of vertices. Since there are less than  $n^2$ many pairs,   for any $H\in \mathbb G(\DD {+-})\cup \mathcal G(D) $
there are less than  $n^2$ many $G\in \mathbb G(\DD {++})$ such that $G'=H$. So we proved $|\mathbb G(\DD {++})|\le n^2\cdot\big( |\mathbb G(\DD {+-})|+|\mathcal G(D)|\big).$

The inequality $|\mathbb G(\DD {--})|\big)\le n^2\cdot\big( |\mathbb G(\DD {+-})|+|\mathcal G(D)|\big)$ can be proved analogously

\medskip
\noindent (b)
Assume that $G\in  \mathcal G(D+1^{+i}_{+i}) \subset  \mathbb G(\DD {+2})$ for some $1\le i\le n$.

We will define a graph $G'\in \mathbb G(\DD {+-})$ such that the symmetric difference of $E(G)$ and $E(G')$ is one edge. Since $\deg_G(v_i)=d_i+2>0$, there is $j$ such that $(v_i,v_j)\in E(G)$. Then $G'=G-(v_i,v_j)$ meets the requirements because
$\mathbf d(G')=\mathbf d(G)+1^{-i}_{-j}=D+1^{+i}_{+i}+1^{-i}_{-j}=D+1^{+i}_{-j}$.

From $G$ you can get back $G'$ provided you know the  symmetric difference of $E(G)$ and $E(G')$, which is just one pair of vertices. Since there are less than  $n^2$ many pairs,
 for any $H\in \mathbb G(\DD {+-}) $
there are less than  $n^2$ many $G\in \mathbb G(\DD {+2})$ such that $G'=H$.
So we proved $|\mathbb G(\DD {+2})|\le n^2\cdot |\mathbb G(\DD {+-})|$.

The inequality $|\mathbb G(\DD {-2})|\le n^2\cdot |\mathbb G(\DD {+-})|$
can be proved analogously.

\medskip\noindent
(c) For each $G\in \mathbb G(\DD {+-})$ we will find a $G'\in \mathbb G(\DD {++})$ such that the symmetric difference of $E(G)$ and $E(G')$ is either an edge, or a path of length 3. Fix $i, j$ such that  $G\in \mathcal G(D+1^{+j}_{-i})$. If there is $k\ne i,j$ such that $v_iv_k$ is not an edge, then let $$G'=G+(v_i,v_j).$$ Then $\mathbf d(G')=1^{+j}_{+k}$, so $G'$ satisfies the requirements.

So we can assume that $v_iv_k\in E(G)$ for $k\ne i,j$. Since the $\deg_G(v_i)=d_i-1 \le n-2$, we obtain that $(v_i,v_j)\notin E(G)$. Let $X=\Gamma_G(v_j)$ and write $d=|X|$. Since $\deg_G(v_j)=d_j+1>0$, we have $d\ge 1$.

\medskip
\noindent {\bf  Claim.} {\em There is $v_k\in X$ such that $(v_k,v_\ell)\notin E(G)$ for some $\ell\ne k$.}

\begin{proof}[Proof of the claim]
Assume on the contrary that  $(v_k,v_\ell)\in E(G)$  for each  $v_k\in X$ and $\ell \ne k$. Thus $\deg_G(v_k)=n-1$ for each $v_k\in X$. Since $v_i,v_j\notin X$, it follows that $d_k=\deg_G(v_k)=n-1$ for each $v_k\in X$.
	
Let $H$ be a realization of the graphic sequence $D$ on the vertex set $\{v_1,\dots, v_n\}$. Since $\deg_H(v_k)=d_k=n-1$ for each $v_k\in X$,  and $v_j\notin X$, it follows that $d_j=\deg_H(v_j)\ge |X|=d$. But $|X|=\deg_G(v_j)=d_j+1$, so $d_j<|X|$. Contradiction,  we proved the Claim.
\end{proof}

By the Claim, we can fix  $v_k\in X$ such that $(v_k,v_\ell)\notin E(G)$ for some
$\ell\ne k$.  Then  $\ell\ne i,j$ since $(v_i,v_k)$ and $(v_j,v_k)$ are edges in $G$. Let
\begin{displaymath}
G'=G+(v_i,v_j)-(v_j,v_k)+(v_k,v_\ell).
\end{displaymath}
Then $\mathbf d(G')=D+1^{+j}_{+\ell}$. So $G'$ satisfies the requirements.
Thus we could always define $G'$.

Since there are less than  $n^2$ edges and less than $n^4$ many paths of length 3, for any $H\in \mathbb G(\DD {++}) $ there are less than  $n^2+n^4$ many $G\in \mathbb G(\DD {+-})$ such that $G'H$. So we proved $|\mathbb G(\DD {+-})|\le (n^4+n^2)\cdot |\mathbb G(\DD {++})| $.

The inequality $|\mathbb G(\DD {+-})|\le (n^4+n^2)\cdot |\mathbb G(\DD {--})|
$ can be proved analogously.
\end{proof}

\begin{remark*}
If $D$ is the constant 0 sequence of length $n$, then $|\mathbb G(\DD {++})|>0$,
but   $|\mathbb G(\DD {+-})|=0$ because $\DD {+-}=\emptyset$, so in \ref{tm:appendix}.\eqref{gdpp}, we can not omit  $|\mathcal G(D)|$
from the RHS of the inequality.
\end{remark*}	
\end{document}